\documentclass[11pt]{article}
\usepackage{amssymb}
\usepackage{amsthm}
\usepackage{amsmath}
\usepackage{amscd}
\newtheorem{theorem}{Theorem}[section]

\newtheorem{lemma}[theorem]{Lemma}
\newtheorem{proposition}[theorem]{Proposition}
\theoremstyle{definition}
\newtheorem{definition}[theorem]{Definition}
\newtheorem{remark}[theorem]{Remark}
\newtheorem{example}[theorem]{Example}

\begin{document}

\centerline{\Large \bf Morse Theory for C*-algebras:}
\centerline{\large \bf A Geometric Interpretation of Some
Noncommutative Manifolds }

 \vline

 \centerline{{\bf Vida Milani}
 \footnote {corresponding author: Vida Milani e-mail:
 v-milani@cc.sbu.ac.ir}}

 \centerline{\small \it Dept. of Math.,
Faculty of Math. Sci., Shahid Beheshti University, Tehran, IRAN}

\centerline{\small \it School of Mathematics, Georgia Institute of
Technology, Atlanta GA, USA}

\centerline{\small \it e-mail: v-milani@cc.sbu.ac.ir}
\centerline{\small \it e-mail: vmilani3@math.gatech.edu}

\centerline{\bf Seyed M.H. Mansourbeigi}

 \centerline{\small \it Dept.
of Electrical Engineering, Polytechnic University, NY, USA}

\centerline{\small \it e-mail: s.mansourbeigi@ieee.org}

\centerline{\bf Ali Asghar Rezaei}

 \centerline{\small \it Dept. of
Math., Faculty of Math. Sci., Shahid Beheshti University, Tehran,
IRAN}

 \centerline{\small \it e-mail:
 A\_Rezaei@sbu.ac.ir}

\begin{abstract}

 The approach we present is a modification of the Morse theory for
 unital
 C*-algebras. We provide tools for the geometric interpretation of noncommutative
 CW complexes. These objects were introduced and studied in [2], [7] and [14]. Some examples
 to illustrate these geometric information in practice are given. A
 classification of unital C*-algebras by noncommutative CW complexes and the modified Morse functions on them
 is the main object of this work.
\end{abstract}


\emph{AMS subject class: 06B30, 46L35, 46L85, 55P15, 55U10}

\newpage

\section{Introduction}

Among the various approaches in the study of smooth manifolds by
the tools from calculus, is the Morse theory. The classical morse
theory provides a connection between the topological structure of
a manifold $M$ and the topological type of critical points of an
open dense family of functions $f: M \rightarrow \mathbb{R}$ (the
Morse functions).

 On a smooth manifold $M$, a point $a \in M$
is {\it a critical point} for a smooth function $f: M \rightarrow
\mathbb{R}$, if the induced map $f_*: T_a(M) \rightarrow
\mathbb{R}$ is zero. The real number $f(a)$ is then called {\it a
critical value}. The function $f$ is {\it a Morse function} if i)
all the critical values are distinct and ii) its critical points
are non degenerate, i.e. the Hessian matrix of second derivatives
at the critical points has a non vanishing determinant. The number
of negative eigenvalues of this Hessian matrix is {\it the index
of
f} at the critical point. The classical Morse theory states as [13]\\

{\bf Theorem }: {\it There exists a Morse function on any
differentiable manifold and any differentiable manifold is a CW
complex with a $\lambda$-cell for each critical point of index
$\lambda$}.

So once we have information around the critical points of a Morse
function on $M$, we can reconstruct $M$ by a sequence of
surgeries.\\

A C*-algebraic approach which links between operator theory and
algebraic geometry, is via a suitable set of equivalence classes
of extensions of commutative C*-algebras. This provides a functor
from locally compact spaces into abelian groups [7], [11], [14].

If $J$ and $B$ are two C*-algebras, an extension of $B$ by $J$ is
a C*-algebra $A$ together with morphisms $j: J \rightarrow A$ and
$\eta : A \rightarrow B$ such that there is an exact sequence

\begin{equation}
\begin{CD}
 0@>>> J
@>{j}>> A @>{\eta}>> B\\
\end{CD}
\end{equation}

The aim of the extension problem is the characterization of those
C*-algebras $A$ satisfying the above exact sequence. This has
something to do with algebraic topology techniques. In the
construction of a CW complex, if $X_{k-1}$ is a suitable
subcomplex, $I^k$ the unit ball and $S^{k-1}$ its boundary, then
the various solutions for the extension problem of $C(X_{k-1})$ by
$C_0(I^k-S^{k-1})$ correspond to different ways of attaching $I^k$
to $X_{k-1}$ along $S^{k-1}$, which means that in the disjoint
union $X_{k-1} \cup I^k$ we identify points $x \in S^{k-1}$ with
their image $\varphi _{k}(x)$ under some attaching map $\varphi
_{k}: S^{k-1} \rightarrow X_{k-1}$.

After the construction of noncommutative geometry [1], there have
been attempts for the formulation of classical tools of
differential geometry and topology in terms of C*-algebras (in
some sense the dualization of the notions, [3], [4], [10]. The
dual concept of CW complexes , with some regards, is the notion of
noncommutative CW complexes [7] and [14]. Our approach in this
work is the geometric study of these structures. So many works are
done on the combinatorial structures of noncommutative simplicial
complexes and their decompositions, for example [2], [6], [8],
[9]. Following these works, together with some topological
constructions , we show how a modification of the classical Morse
theory to the level of C*-algebras will provide an innovative way
to explain the geometry of noncommutative CW complexes through the
critical ideals of the modified Morse function. This leads to some
classification theory.

This paper is prepared as follows. After an introduction to the
notion of primitive spectrum of a C*-algebra, it will proceed the
topological structure in detail and present some instantiation. In
the continue we study the noncommutative CW complexes and
interpret their geometry by introducing the modified Morse
function. All these provide tools for the modified Morse theory
for C*-algebras. The last section devoted to the prove of this
theorem.
It states as\\

 {\bf Main Theorem }: {\it Every unital C*-algebra with an
 acceptable Morse function on it is of pseudo-homotopy type of a
 noncommutative CW complex, having a k-th decomposition cell
 for each critical chain of order k}.\\

\section{ The Structure of the Primitive Spectrum}

The technique we follow to link between the geometry, topology and
algebra is the primitive spectrum point of view. In fact as we
will see in our case it is a promissive candidate for the
noncommutative analogue of a topological manifold $M$. We review
some preliminaries on the primitive spectrum. Details can be found
in [5], [10], [12].

Let $M$ be a compact topological manifold and $A = C(M)$ be the
commutative unital C*-algebra of continuous functions on $M$. The
{\it primitive spectrum} of $A$ is the space of kernels of
irreducible *-representations of $A$. It is denoted by $Prim(A)$.
The topology on this space is given by the closure operation as
follows:

For any subset $X \subseteq Prim(A)$, the closure of $X$ is
defined by
\begin{equation}
{\bar X} := \{ I \in Prim(A) : \bigcap_{J \in X} J \subset I \}
\end{equation}
This operation defines a topology on $Prim(A)$ (the hull-kernel
topology), making it into a $T_0$-space.

\begin{definition}
A subset $X \subset Prim(A)$ is called {\it absorbing} if it
satisfies the following condition

\begin{equation}
I \in X , I \subseteq J \Rightarrow J \in X
\end{equation}
\end{definition}

\begin{lemma}
The closed subsets of $Prim(A)$ are exactly its absorbing subsets.
\end{lemma}

\begin{proof}
It is clear from the very definition of closed sets.
\end{proof}

For each $x \in M$ let
$$I_x := \{ f \in A : f(x) = 0 \}$$
$I_x$ is a closed maximal ideal of $A$. It is in fact the kernel
of the evaluation map
$$(ev)_x : A \rightarrow \mathbb{C}$$
$$f \rightsquigarrow f(x)$$

This provides a homeomorphism

\begin{equation}
I : M \rightarrow Prim(A)
\end{equation}

between $M$ and $Prim(A)$, defined by $I(x) := I_x$.

To each $I \in Prim(A)$, there corresponds an absorbing set
$$W_I := \{ J \in Prim(A) : J \supseteq I \}$$
and an open set
$$O_I := \{ J \in Prim(A) : J \subseteq I \}$$
containing $I$.

Being a $T_0$-space, $Prim(A)$ can be made into a partially
ordered set (poset) by setting for $I,J \in Prim(A)$,
$$I < J \Leftrightarrow I \subset J$$
which is equivalent to $O_I \subset O_J$ or $W_I \supset W_J$.

The topology of $Prim(A)$ can be given equivalently by means of
this partial order
$$I < J \Leftrightarrow J \in {\bar {\{ I \}}}$$
where ${\bar {\{ I \}}}$ is the closure of the one point set
$\{I\}$.

Now let $A$ be an arbitrary unital C*-algebra. Since $A$ is
unital, then $Prim(A)$ is compact. Let $Prim(A) = \bigcup_{i=1}^n
O_{I_i}$ be a finite open covering.

An equivalence relation on $Prim(A)$ is given by
$$I \sim J \Leftrightarrow J \in O_I (I \in O_J)$$

In each $O_{I_i}$ choose one $I_i$ with respect to the above
equivalence relation.

Let $I_1,I_2,...I_m$ be chosen this way so that $Prim(A)$ is made
into a finite lattice for which the points are the equivalence
classes of $[I_1],...[I_m]$. For simplicity we show each class
$[I_i]$ by its representative $I_i$. Let
$$J_{i_1,....i_k} := I_{i_1} \cap ... \cap I_{i_k}$$
where $1 \leq i_1,...i_k \leq m , 1 \leq k \leq m$

Set
$$W_{i_1,....i_k} := \{ J \in Prim(A) : J \supset J_{i_1,....i_k}
\}$$

This is a closed subset of $Prim(A)$.

\begin{definition}
 When $J_{i_1,....i_k} \neq 0$, then it is called a {\it
k-ideal} in $A$ and its corresponding closed set
$W_{i_1,....i_k}$ in $Prim(A)$ is called a ${\it k-chain}$.\\

\end{definition}

\begin{remark}
 If for some $1 \leq i_1,...i_k \leq m , 1 \leq k \leq m$, we have
$J_{i_1,....i_k} = 0$, then $W_{i_1,....i_k} = Prim(A)$. Also for
each pair of indices $(i_1,....i_k) , \sigma (i_1,....i_{k+1})$
$$W_{i_1,....i_k} \subseteq W_{\sigma(i_1,....i_{k+1})}$$
where $\sigma$ is a permutation on k+1 elements.
\end{remark}

\begin{remark}
 Let $$X_0 \subset X_1 \subset ... \subset X_n=X$$
be an n-dimensional CW complex structure for a topological space
$X$, so that $X_0$ is a finite discrete space consisting of
0-cells, and for $k=1,...,n$ each k-skeleton $X_k$ is obtained by
attaching $\lambda_k$ number of k-disks to $X_{k-1}$ via the
attaching maps
$$\varphi _k : \bigcup_{\lambda _k} S^{k-1} \rightarrow X_{k-1}$$

In other words
\begin{equation}
X_k = \frac{X_{k-1} \bigcup (\cup _{\lambda _k} I^k)}{x \sim
\varphi _k (x)} := X_{k-1} \bigcup_{\varphi _k}(\cup _{\lambda _k}
I^k)
\end{equation}
wherever $x \in S^{k-1}$, where $I^k := [0,1]^k$ and $S^{k-1} :=
\partial I^k$. The quotient map is denoted by
$$\rho :
X_{k-1} \bigcup (\cup_{\lambda_k}I^{k}) \rightarrow X_k$$

 For more details see [11].

A cell complex structure is induced on $Prim(C(X))$ by the
following procedure:

Let $A_k=C(X_k)$, $k=1,...,n$. For each 0-cell $C_0$ in $X_0$, let
$I_{C_0}$ be its image under the homeomorphism $I : X_0
\rightarrow Prim(C(X_0))$ of relation (4). By considering the
restriction of functions on $X$ to $X_0$, $I_{C_0}$s will be the
0-ideals for $A = C(X)$ and
$$W_{C_0} := \{ J \in Prim(A) : I_{C_0} \subset J \}$$
the 0-chains for $Prim(A)$.

The 1-ideals are of the form $I_{C_1} := \bigcap_{x \in C_1} I_x$
with the corresponding 1-chains
$$W_{C_1} := \{ J \in Prim (A) : I_{C_1} \subset J \}$$

In the same way the k-ideals are $I_{C_k} = \bigcap_{x \in C_k}
I_x$ for $2\leq k \leq n$, with the corresponding k-chains
$$W_{C_k} := \{ J \in Prim(A) : I_{C_k} \subset J \}$$

(An ideal in $A_{k-1}$ can be thought of as an ideal in $A_k$ by
the restriction of functions.)\\

In the following two examples we identify the k-ideals and the
k-chains for the CW complex structures of the closed interval
[0,1] and the 2-torus $S^1 \times S^1$.
\end{remark}

\begin{example}
Let $X_0 = \{ 0,1 \}$ and $X_1 = [0,1]$ be the zero and one
skeleton for a CW complex structure of [0,1]. $A_0 = C(X_0)
\simeq \mathbb{C} \oplus \mathbb{C}$ and $A_1 = C(X_1)$ and the
0-ideals $I_0$ and $I_1$ and their corresponding 0-chains $W_0$
and $W_1$ are
$$I_0 = \{ f \in A_0 : f(0) = 0 \} \simeq \mathbb{C},
I_1 = \{ f \in A_0 : f(1) = 0 \} \simeq \mathbb{C}$$

and
$$W_0 = \{J \in PrimA_0 : I_0 \subset J \} \simeq \{ 0 \},
W_1 = \{J \in PrimA_0 : I_1 \subset J \} \simeq \{ 1 \}$$

For the only 1-ideal we have $$I = I_0 \cap I_1 = 0$$ with the
corresponding 1-chain
$$W_I = \{ J \in Prim(A) : I \subset J \} = Prim(A) \simeq [0,1]$$

\end{example}

\begin{example}
 Let
$$X_0 = \{ 0 \} , X_1 = \{ \alpha , \beta \}
, X_2 = T^2 = S^1 \times S^1$$ be the skeletons for a CW complex
structure for the 2-torus $T^2$. $\alpha , \beta$ are homeomorphic
images of $S^1$ (closed curves with the origin 0). Let $A_0 =
C(X_0) = \mathbb{C}$ , $A_1 = C(X_1)$ and $A_2 = A = C(T^2)$.

The 0-ideal and its corresponding 0-chain are
$$I_0 = \{ f \in A_0 : f(0) = 0 \}$$
$$W_0 = \{ J \in Prim(A_0) : I_0 \subset J \} \simeq Prim(A_0) =
\{ 0 \}$$ Also the 1-ideals $I_1 , I_2$ and 1-chains $W_{I_1} ,
W_{I_2}$ are
$$I_1 = \{ f \in A_1 : f( \alpha ) = 0 \} = \cap_{x \in \alpha}
I_x$$
$$I_2 = \{ f \in A_1 : f( \beta ) = 0 \} = \cap_{x \in \beta}
I_x$$
$$W_{I_1} = \{ J \in Prim(A_1) : I_1 \subset J \} \simeq \alpha$$
$$W_{I_2} = \{ J \in Prim(A_1) : I_2 \subset J \} \simeq \beta$$

Finally the only 2-ideal and its corresponding 2-chain are
$$I = \{ f \in A : f(T^2) = 0 \} \simeq 0$$
$$W_I = \{ J \in Prim(A) : I \subset J \} \simeq T^2$$

\end{example}

\section{The Noncommutative CW Complexes (NCCW Complexes)}

In this section we see how the construction of the primitive
spectrum of the previous section helps us to study the
noncommutative CW complexes.\\

For a continuous map $\phi : X \rightarrow Y$ between topological
spaces $X$ and $Y$, the C*-morphism induced on their associated
C*-algebras is denoted by $C(\phi) : C(Y) \rightarrow C(X)$ and is
defined by $C(\phi)(g):= g o \phi$ for $g \in C(Y)$.\\

\begin{definition}
Let $A_1$, $A_2$ and $C$ be C*-algebras. {\it A pull back for $C$
via morphisms $\alpha_1 : A_1 \rightarrow C$ and $\alpha_2 : A_2
\rightarrow C$} is the C*-subalgebra of $A_1 \oplus A_2$ denoted
by $PB(C, \alpha_1, \alpha_2)$ defined by
$$PB(C, \alpha_1, \alpha_2):= \{a_1 \oplus a_2 \in A_1 \oplus A_2
: \alpha_1(a_1) = \alpha_2(a_2)\}$$

\end{definition}

For any C*-algebra $A$, let
$$S^nA := C(S^n \rightarrow A) , I^nA := C([0,1]^n \rightarrow A)
, I_0^nA := C_0((0,1)^n \rightarrow A)$$ where $S^n$ is the
n-dimensional unit sphere.\\

We review the definition of noncommutative CW complexes from [7],
[14].
\begin{definition}
A {\it 0-dimensional noncommutative CW complex} is any finite
dimensional C*-algebra $A_0$. Recursively an {\it n-dimensional
noncommutative CW complex (NCCW complex)} is any C*-algebra
appearing in the following diagram
\end{definition}

\begin{equation}
\begin{CD} 0@>>>{I}_0^nF_n@>>> A_n
@>{\pi}>> A_{n-1} @>>> 0\\
&&@|@VV{f_n}V@VV{\varphi_n}V\\
0@>>>{I}_0^nF_n@>>>{I}^nF_n @>{\delta}>> S^{n-1}F_n @>>> 0\end{CD}
\end{equation}

Where the rows are extensions, $A_{n-1}$ an (n-1)-dimensional
noncommutative CW complex, $F_n$ some finite dimensional
C*-algebra of dimension $\lambda_n$, $\delta$ the boundary
restriction map, $\varphi_n$ an arbitrary morphism (called the
connecting morphism), for which
\begin{equation}
A_n = PB(S^{n-1} F_n, \delta, \varphi_n):= \{ ( \alpha , \beta )
\in I^n F_n \oplus A_{n-1} : \delta ( \alpha ) = \varphi_n( \beta
) \}
\end{equation}

$f_n$ and $\pi$ are respectively projections onto the first and
second coordinates.

With these notations $\{ A_0,...,A_n \}$ is called {\it the
noncommutative CW complex decomposition of dimension $n$ for
$A=A_n$}

For each $k=0,1,...n$, $A_k$ is called {\it the k-th decomposition
cell.}

\begin{proposition}
Let $X$ be an n-dimensional CW complex containing cells of each
dimension $\leq n$. Then there exists a noncommutative CW complex
decomposition of dimension n for $A = C(X)$.

Conversely suppose $\{A_0,...,A_n\}$ be a noncommutative CW
complex decomposition of dimension n for the C*-algebra $A$ such
that $A$ and all the $A_is$ ($i=0,..,n)$ are unital. For each $k
\leq n$, let $X_k = Prim(A_k)$. Then there exists an n-dimensional
CW complex structure on $Prim(A)$ with $X_k$ as its k-skeleton for
each $k \leq n$.
\end{proposition}

\begin{proof}
Let
$$X_0 \subset X_1 \subset ... \subset X_n = X$$
be a CW complex structure for $X$ where for each $k \leq n$, $X_k$
is the k-skeleton defined in relation (5). For each $ k=0,...,n$,
let $A_k = C(X_k)$, $i : \bigcup_{\lambda_k}S^{k-1} \rightarrow
\bigcup_{\lambda_k}I^{k}$ be the injection, and $\varphi_k :
\bigcup_{\lambda_k}S^{k-1} \rightarrow X_{k-1}$ be the attaching
maps. Furthermore let $C(i)$ and $C(\varphi)$ be their induced
maps. Let $$PB : = PB(C(\bigcup_{\lambda_k}S^{k-1}), C(\varphi_k),
C(i))$$

Define
$$\theta : C(X_k) \rightarrow PB$$ by $\theta (f) = (fo \rho)_1 \oplus (fo
\rho)_2$ for $f \in C(X_k)$.

Where $(fo \rho)_1$ is the restriction of $(fo \rho)$ to
$\bigcup_{\lambda_k}I^k$ and $(fo \rho)_2$ is the restriction of
$(fo \rho)$ to $X_{k-1}$.\\

$\theta$ is well defined since $C(\varphi_k)((fo \rho)_1) =
 C(i)((fo \rho)_2)$.
Also for $(h,g) \in PB$, we have $C(\varphi_k)(h) = C(i)(g)$ and
so if $f \in C(X_k)$ be defined by $f(y) = g(y)$ for $y \in
\bigcup_{\lambda_k}I^{k}$ and $f(y) = h(y)$ for $y \in X_{k-1}$,
then $\theta (f) = (h,g)$.\\

Now the noncommutative CW complex decomposition of dimension n for
$A = C(X)$ is given by $\{A_0,...,A_n\}$.

Conversely let $A_n$ be as in (7). Let
$$\varphi_n^* : S^{n-1} \rightarrow Prim(A_{n-1})$$
be the attaching map induced by the connecting morphism
$$\varphi_n : A_{n-1} \rightarrow S^{n-1}F_n$$
of diagram (6). Then using the notation in relation (5),
$$Prim(A_n) = Prim(A_{n-1}) \cup_{\varphi_n^*} I^n$$

We note that $\varphi_n^* = C(\varphi_n)$.
 Furthermore for $k \leq n$, $\varphi_k^*(S^{k-1})$
is a closed subset of $Prim(A_{k-1})$. It is of the form
$$\varphi_k^*(S^{k-1}) = \{ J \in Prim (A_{k-1}) : I_{k-1} \subset J
\}$$ for some ideal $I_{k-1}$ in $A_{k-1}$. In fact
$$I_{k-1} = \bigcap_{J \in \varphi_k^*(S^{k-1})} J$$
\end{proof}

\begin{example}

 Following the notations of diagram (6), a
1-dimensional noncommutative CW complex decomposition for
$A=C([0,1]) = C(I)$ is given by
$$A_0 = \mathbb{C} \oplus \mathbb{C} , A_1 = C([0,1])$$

Let $F_1 = \mathbb{C}$, then
$$I_0^1F_1 = C_0((0,1)) , I^1F_1 = C([0,1]) , S^0F_1 = \mathbb{C}
\oplus \mathbb{C}$$  $\varphi_1 = id$. Also
$$C(I) =  PB(S^0F_1, \delta, \varphi_1) = \{ f \oplus ( \lambda \oplus \mu ) \in C([0,1]) \oplus
(\mathbb{C} \oplus \mathbb{C}) : f(0) = \lambda , f(1) = \mu \}$$
together with the maps

$\pi : A_1 \rightarrow A_0$ defined by $\pi (f \oplus ( \lambda
\oplus \mu )) = \lambda \oplus \mu$ and $f_1 : A_1 \rightarrow
I^1F_1 = A_1$ defined by $f_1 (f \oplus ( \lambda \oplus \mu )) =
f$ and finally $\delta : I^1F_1 = A_1 \rightarrow S^0F_1 =
\mathbb{C} \oplus \mathbb{C}$ defined by $\delta (f) = f(0) \oplus
f(1)$.
\end{example}

\section{Modified Morse Theory on C*-Algebras}

In this section, following the study of the Morse theory for the
cell complexes in [2], [6], [8], [9], with some modification, we
define the Morse function for the C*-algebras and state and prove
the modified Morse theory for the noncommutative CW complexes.
This is a classification theory in the category of C*-algebras and
noncommutative CW complexes.

\begin{definition}
Let $A$ and $B$ be C*-algebras. Two morphisms $\alpha , \beta : A
\rightarrow B$ are said to be {\it homotopic} if there exists a
family $\{H_t\}_{t \in [0,1]}$ of morphisms $H_t : A \rightarrow
B$ such that for each $a \in A$ the map $t \mapsto H_t(a)$ is a
norm continuous path in $B$ with $H_0 = \alpha$ and $H_1 = \beta$.
In this case we write $\alpha \sim \beta$.

C*-algebras $A$ and $B$ are called {\it of the same homotopy type
} if there are morphisms $\varphi : A \rightarrow B$ and $\psi : B
\rightarrow A$ such that $\varphi o \psi \sim id_{B}$ and $\psi o
\varphi \sim id_{A}$. In this case the morphisms $\varphi$ and
$\psi$ are called {\it homotopy equivalent}.
\end{definition}

\begin{definition}
Let $A$ and $B$ be unital C*-algebras. We say $A$ is {\it of
pseudo-homotopy type} of $B$ if $C(Prim(A))$ and $B$ are of the
same homotopy type.
\end{definition}

\begin{remark}
In the case of unital commutative C*-algebras, by the GNS
construction, $C(Prim(A)) = A$, [10]. So the notions of
pseudo-homotopy type and the same homotopy type are equivalent.
\end{remark}

Let $A$ be a unital C*-algebra and
$$ \Sigma = \{\{ W_{i_1,...,i_k} \}_{1 \leq k \leq n}\}_{1 \leq i_1,...,i_k \leq
n}$$ be the set of all k-chains (k=1,...,n) in $Prim(A)$.\\

\begin{lemma}
Let
 $$\Gamma = \{\{ I_{i_1,...,i_k} \}_{1 \leq k \leq n}\}_{1 \leq
i_1,...,i_k \leq n}$$
be the set of all k-ideals corresponding to
the k-chains of $\Sigma$ for k=1,...,n, then $\Gamma$ is an
absorbing set.
\end{lemma}

\begin{proof}
This follows from the fact that for each $ I_{i_1,...,i_k}\in
\Gamma$ and for each $J \in \Gamma$, the relation $I_{i_1,...,i_k}
\subset J$ is equivalent to $J =  I_{i_1,...,i_t}$ for some $t
\leq k$ meaning $J \in \Gamma$.
\end{proof}

\begin{definition}
Let $f : \Sigma \rightarrow \mathbb{R}$ be a function. The k-chain
$W_k = W_{i_1,...,i_k}$ is called {\it a critical chain of order
k} for $f$, if for each (k+1)-chain $W_{k+1}$ containing $W_k$ and
for each (k-1)-chain $W_{k-1}$ contained in $W_k$, we have
$$f(W_{k+1}) \geq f(W_k) , f(W_{k-1}) \leq f(W_k)$$
\end{definition}

The corresponding ideal $I_k$ to $W_k$ is called {\it the critical
ideal of order k}.

\begin{definition}
 A function $f : \Sigma \rightarrow
\mathbb{R}$ is called {\it a modified Morse function} on the
C*-algebra $A$, if for each k-chain $W_k$ in $\Sigma$, there is at
most one (k+1)-chain $W_{k+1}$ containing $W_k$ and at most one
(k-1)-chain $W_{k-1}$ contained in $W_k$, such that
$$f(W_{k+1}) \leq f(W_k) , f(W_{k-1}) \geq f(W_k)$$

$f$ is called {\it acceptable Morse function} if for each $k$, if
$f$ has a critical chain of order $k$, then there exists critical
chains of order $i$ for all $i\leq k$.

\end{definition}

Now we state our main theorem. This geometric condition for a
C*-algebra to admit a noncommutative CW complex decomposition
classifies specific unital C*-algebras up to pseudo-homotopy type.

\begin{theorem}
Every unital C*-algebra $A$ with an acceptable modified Morse
function $f$ on it, is of pseudo-homotopy type of a noncommutative
CW complex having a k-th decomposition cell for each critical
chain of order k.

\end{theorem}

Before starting the proof of this theorem, we state the discrete
Morse theory of Forman from [8] and state our modification of
it.\\

{\bf Theorem (Discrete Morse Theory)}: Suppose $\Delta$ is a
simplicial complex with a discrete Morse function. Then $\Delta$
is homotopy equivalent to a CW complex with one cell of dimension
p for each critical p-simplex [8].\\

\begin{lemma}
If $f$ is an acceptable modified Morse function on $A$, then
$Prim(A)$ is homotopy equivalent to a CW complex with exactly one
cell of dimension p for each critical chain of order p.

\end{lemma}

\begin{proof}

In the discrete Morse theory it suffices to substitute $\Gamma$
 for the simplicial complex $\Delta$. Since $\Gamma$ is absorbing,
 it satisfies the properties of the simplicial complex $\Delta$
 in the discrete Morse theorem. It follows that $Prim(A)$ is
 homotopy equivalent to a CW complex with exactly one cell of
 dimension p for each critical chain of order p.
\end{proof}

Now we start the proof of the main theorem.

\begin{proof}
When $A$ is a unital C*-algebra, then the acceptable modified
Morse function on $A$ is in fact a function on the simplicial
complex of all k-ideals in $Prim(A)$ (a function on $\Gamma$).
From lemma 4.8 we conclude that $Prim(A)$ is of homotopy type of a
finite CW complex $\Omega$. From the proposition 3.3 there is a
noncommutative complex decomposition for $C(\Omega)$ making it
into a noncommutative CW complex. Now $C(Prim(A))$and $C(\Omega)$
are C*-algebras of the same homotopy type, which means $A$ is of
psuedo-homotopy type of the noncommutative CW complex $C(\Omega)$.
Furthermore since $f$ is acceptable, from the proof of proposition
3.3 it follows that when there exists a critical k-chain for $f$,
then there exists C*-algebras $A_i$ for each $i \leq k$ so that
$\{ A_0,...,A_k \}$ is a noncommutative CW complex decomposition
for $C(Prim(A))$ yielding a noncommutative CW complex
decomposition for $C(\Omega)$ .

\end{proof}

\end{document}